\documentclass{article}
\usepackage{fullpage}
\usepackage{setspace}
    \usepackage{graphicx}
    \usepackage{amsmath, amssymb, latexsym, amsthm, url}
    \usepackage{mathrsfs,color,fancyhdr}
    \usepackage{cite}
    \usepackage{prettyref}

    \newcommand{\dens} {\bar{\vspace{0.2in}{\includegraphics[width=0.09in]{dens}}\vspace{0.2in}}}
    \newcommand{\one}{\mathrm{1}}

      \newtheorem{thm}{Theorem}[section]
      \newcommand{\bthm}{\begin{thm}} \newcommand{\ethm}{\end{thm}}
      \newtheorem{prop}[thm]{Proposition}
      \newcommand{\bprp}{\begin{prop}} \newcommand{\eprp}{\end{prop}}
      \newtheorem{fact}[thm]{Fact}
      \newcommand{\bfct}{\begin{fact}} \newcommand{\efct}{\end{fact}}
      \newtheorem{prob}[thm]{Problem}
      \newcommand{\bprb}{\begin{prob}} \newcommand{\eprb}{\end{prob}}
      \newtheorem{lem}[thm]{Lemma}
      \newcommand{\blem}{\begin{lem}} \newcommand{\elem}{\end{lem}}
      \newtheorem{claim}[thm]{Claim}
      \newcommand{\bclm}{\begin{claim}} \newcommand{\eclm}{\end{claim}}
      \newtheorem{cor}[thm]{Corollary}
      \newcommand{\bcor}{\begin{cor}} \newcommand{\ecor}{\end{cor}}
      \newtheorem{conj}[thm]{Conjecture}
      \newcommand{\bcnj}{\begin{conj}} \newcommand{\ecnj}{\end{conj}}
      \theoremstyle{definition}
      \newtheorem{defn}[thm]{Definition}
      \newcommand{\bdfn}{\begin{defn}} \newcommand{\edfn}{\end{defn}}
      \theoremstyle{remark}
      \newtheorem{rem}[thm]{Remark}
      \newcommand{\brem}{\begin{rem}} \newcommand{\erem}{\end{rem}}
      \newtheorem{cnv}[thm]{Convention}
      \newcommand{\bcnv}{\begin{cnv}} \newcommand{\ecnv}{\end{cnv}}
      \newtheorem{exam}[thm]{Example}
      \newcommand{\bexm}{\begin{exam}} \newcommand{\eexm}{\end{exam}}
      \newcommand{\bpf}{\begin{proof}} \newcommand{\epf}{\end{proof}}
      \newtheorem{exer}[thm]{Exercise}
      \newcommand{\bexer}{\begin{exer}} \newcommand{\eexer}{\end{exer}}
      \newcommand{\be}{\begin{enumerate}}
      \newcommand{\ee}{\end{enumerate}}
      \newcommand{\bi}{\begin{itemize}}
      
      \newcommand{\ei}{\end{itemize}}
      \newtheorem{obs}[thm]{Observation}

        \newcommand{\Z}{{\mathbb Z}}
        \newcommand{\N}{{\mathbb N}}

        \newcommand{\E}{\mathbb{E}}
        \newcommand{\PP}{\mathbb{P}}

        \newcommand{\Om}{\Omega}
        \newcommand{\quenchedP}{P}
        \newcommand{\annealedP}{\mathbb P}
       \newcommand{\ignore}[1]{}
       \newcommand{\Unif}{\mathcal {U}}

\usepackage{color}

\definecolor{Red}{rgb}{1,0,0}
\definecolor{Blue}{rgb}{0,0,1}
\definecolor{Olive}{rgb}{0.41,0.55,0.13}
\definecolor{Yarok}{rgb}{0,0.5,0}
\definecolor{Green}{rgb}{0,1,0}
\definecolor{MGreen}{rgb}{0,0.8,0}
\definecolor{DGreen}{rgb}{0,0.55,0}
\definecolor{Yellow}{rgb}{1,1,0}
\definecolor{Cyan}{rgb}{0,1,1}
\definecolor{Magenta}{rgb}{1,0,1}
\definecolor{Orange}{rgb}{1,.5,0}
\definecolor{Violet}{rgb}{.5,0,.5}
\definecolor{Purple}{rgb}{.75,0,.25}
\definecolor{Brown}{rgb}{.75,.5,.25}
\definecolor{Grey}{rgb}{.5,.5,.5}

    \newcommand{\cookieenv}{\omega}
    \author{Gideon Amir\footnote{Gideon Amir, Bar Ilan University {\tt gidi.amir@gmail.com }}, Noam Berger\footnote{Noam Berger, The Hebrew University of Jerusalem and Technische Universit\"at M\"unchen.
{\tt berger@math.huji.ac.il}}, Tal Orenshtein\footnote{Tal Orenshtein, Weizmann Institute of Science and Technische Universit\"at M\"unchen. {\tt  tal.orenshtein@weizmann.ac.il}} }

    \date{}

    \title{Zero-one law for directional transience of one dimensional excited random walks}

\begin{document}

    \maketitle

    \begin{abstract}
    The probability that a one dimensional excited random walk in stationary ergodic and elliptic cookie environment is transient to the right (left) is either zero or one. This solves a problem posed by Kosygina and Zerner \cite{kosygina2012excited}. As an application, a law of large numbers holds in these conditions.
\\

\centerline{\bf{R\`{e}sum\`{e}}}
     La probabilit\`{e} q'une marche al\`{e}atoire unidimensionnelle excite dans un environnement ergodique et elliptique soit transiente a gauche ou a droite est soit
    nulle soit un. Ceci r\`{e}sout un probl\`{e}me pose par Kosygina et Zerner. Comme application,
    une loi des grands nombres est valable dans de telles conditions.
    \end{abstract}
\hfil



 \section{Introduction\label{sec:intro} }
    Excited random walk was introduced by Itai~Benjamini and David~B.~Wilson in 2003 \cite{benjamini2003excited}. Later
     the model was generalized by Martin~P.~W.~Zerner \cite{zerner2005multi}. It was studied extensively in recent years by numerous researchers, and an almost up to date account may be found in the recent survey of Kosygina and Zerner \cite{kosygina2012excited}.

The generalized model due to Zerner is informally known as Cookie Walk, recently the term `Brownie Motion' has been gaining popularity among researchers in the field.

The model is defined as follows:
Let $\Omega=[0,1]^{\Z\times\N}$, and endow this space with the standard $\sigma$-algebra (namely the product of Borel $\sigma$-algebras on the intervals). Let $\mu$ be a probability measure on $\Omega$ which is invariant and ergodic with respect to the $\Z$-shift (but not necessarily the $\N$-shift). We call $\mu$ the cookie distribution. We call each $\cookieenv\in\Omega$ a cookie environment. Notationally, $\cookieenv(x,n)\in[0,1]$ is called the $n$-th cookie in the location $x$.

We say that the distribution $\mu$ is {\em elliptic} if $\mu((0,1)^{\Z\times\N})=1$, 
and {\em uniformly elliptic} if there exists $\epsilon>0$ such that $\mu([\epsilon,1-\epsilon]^{\Z\times\N})=1$.

Given a cookie environment $\cookieenv$ and an initial position $x\in\Z$ we define the excited random walk driven by $\cookieenv$:
\begin{eqnarray*}
\quenchedP_{\cookieenv,x}(X_0=x)=1, \\
\quenchedP_{\cookieenv,x}(X_n=X_{n-1}+1\ |\ X_0,X_1,\ldots,X_{n-1}) &=& \cookieenv(X_{n-1},\#\{k\leq n-1: X_k=X_{n-1}\}), \\
\quenchedP_{\cookieenv,x}(X_n=X_{n-1}-1\ |\ X_0,X_1,\ldots,X_{n-1}) &=& 1 - \quenchedP_{\cookieenv,x}(X_n=X_{n-1}+1\ |\ X_0,X_1,\ldots,X_{n-1}).
\end{eqnarray*}
We associate $\mu$ with the annealed, or averaged, distribution defined by
\[
\annealedP_x (\cdot)= \int_\Omega\quenchedP_{\cookieenv,x} (\cdot) d\mu(\cookieenv).
\]

In this paper we are interested in the probability that the random walk is {\em transient to the right}, i.e. that
$\lim_{n\to\infty}X_n=+\infty$. We use $A^+$ to denote the event that the random walk is transient to the right, and $A^-$ to denote the event of transience to the left (i.e. $\lim_{n\to\infty}X_n=-\infty$).

    In their recent survey Kosygina and Zerner raised a version of the following problem (see Problem 3.5 of \cite{kosygina2012excited}):

    \begin{prob}\label{prob:KZ}
    Find conditions on the distribution $\mu$ which imply a zero-one law for a directional transience of one dimensional excited random walk, i.e. conditions which imply that $\annealedP_0(A^+)\in\{0,1\}$.
    \end{prob}

The main result of this paper answers Problem \ref{prob:KZ}, namely:

    \begin{thm}\label{thm:directionalZeroOneLaw}
    Let $\mu$ be a stationary ergodic (with respect to the $\Z$-shift) and elliptic probability measure on the space $\Omega$ of cookie environments. Then $\annealedP_0(A^+)\in\{0,1\}$.
    \end{thm}

\subsection{Law of Large Numbers}
Kosygina and Zerner proved in \cite[Theorem 4.1]{kosygina2012excited} that for a stationary ergodic and elliptic probability measure over cookie environments, if a directional 0-1 law holds then a law of large numbers holds.
Using Theorem \ref{thm:directionalZeroOneLaw} an immediate corollary is the following law of large numbers.

\begin{thm}\label{thm:LLN} Let $\mu$ be a stationary ergodic (with respect to the $\Z$-shift $\theta$) and elliptic probability measure on the space $\Omega$ of cookie environments. Then $\annealedP_0(\lim_{n\to\infty}\frac{X_n}{n}=v)=1$ for some deterministic $v\in[-1,1]$.
\end{thm}

One can write a different, more direct, proof of Theorem~\ref{thm:LLN} by noticing that the stationarity assumption in the proof of Theorem~4.1 in \cite{kosygina2012excited} can be slightly weakened.
We refer the reader to Theorem~4.3 in \cite{Amir2013excited} and the discussion above it for details.

\subsection{Previous work}
In some examples in the literature, special cases of Theorem \ref{thm:directionalZeroOneLaw} are derived as special cases of stronger characterization theorems.

Benjamini and Wilson \cite{benjamini2003excited} showed that whenever a cookie environment $\cookieenv$ satisfies $\omega(x,1)=p$ for all $x\in\Z$ and $\cookieenv(x,i)=\frac{1}{2}$ for all $x\in\Z$ and $i\ge 2$, then the walk is $\PP_\cookieenv$-a.s.\ recurrent for all $p\in (0,1)$.
Zerner \cite{zerner2005multi} showed that if the measure $\mu$ is stationary ergodic and satisfies $\mu([\frac{1}{2},1]^{\Z\times\N})=1$, then it is transient to the right if and only if either $\mu(\cookieenv(0,1)=1)=1$ or $\delta > 1$, where $\delta=\E_\mu(\sum_{i=1}^\infty(2\cookieenv(0,i)-1))$.
Kosygina and Zerner \cite{kosygina2008positively} showed that whenever the measure $\mu$ is i.i.d.\ (that is, the sequence of columns $\cookieenv(x,\cdot)$, $x\in\Z$, is i.i.d.\ under $\mu$), weakly elliptic (that is, $\mu( \prod_{i=1}^M\cookieenv(0,i) ) > 0$ and $\mu( \prod_{i=1}^M (1-\cookieenv(0,i) ) )>0$, and there exists a deterministic number $M$ so that $\mu(\cookieenv(0,i)=\frac{1}{2}))=1$ for all $i>M$ then the walk is transient to the right if and only if $\delta > 1$, and transient to the left if and only if $\delta<-1$.
Kosygina and Zerner \cite{kosygina2012excited} proved a Kalikow-type 0-1 law, i.e. a 0-1 law for (non directional) transience for stationary ergodic and elliptic measure over cookie environments, see Theorem~\ref{thm:KalikowType}.

\subsection{Structure of the paper}
The paper is structured as follows: In Section \ref{sec:preliminaries} we present some concepts and processes that take part in the proof. In section \ref{sec:arrow} we introduce arrow environments, in section \ref{sec:zpluszminus} we introduce two associated processes $Z^+$ and $Z^-$, and in section \ref{sec:survival vs. hitting} we study some of their connections to cookie random walks. In section  \ref{sec:comoz} we study monotonicity and symmetry properties of $Z^+$ and $Z^-$, and present an easy lemma which is, however, the core of our argument. In Section \ref{sec:0-1 Law for Directional Transience}
 we reprove a theorem by Kosygina and Zerner. We do this for two purposes. The first purpose is to keep the paper self-contained, and the second is to enable us to easily use notations and lemmas from their proof in the proof of our main result.
In Section \ref{sec:pfmain} we prove Theorem \ref{thm:directionalZeroOneLaw}.

    \section{Preliminaries}\label{sec:preliminaries}

  In this section we give some basic definitions and lemmas which are necessary for the proof of Theorem \ref{thm:directionalZeroOneLaw}.

  \subsection{Arrow environments}\label{sec:arrow}

  Let $\cookieenv$ be a cookie environment. We can realize $\cookieenv$ into a list of arrows, or instructions, which tell the walker where to walk to in every step of the process. More precisely, let $U=[0,1]^{\Z\times\N}$, and let $F:\Omega\times U\to\{0,1\}^{\Z\times\N}$ be defined as
  $F(\cookieenv,u)(x,n)={\bf 1}_{u(x,n)<\cookieenv(x,n)}$. An element $a\in \{0,1\}^{\Z\times\N}$ is called an \emph{arrow environment}.

  We now endow $U$ with the standard Borel $\sigma$-algebra $\mathcal{B}_U$, and with the product measure $\PP_U$ of $\Unif[0,1]$ distributions. The following lemma is a standard Ergodic theoretic fact. For convenience, a proof sketch of this fact may be found in the Appendix.
  \begin{lem}\label{lem:arrow_ergod}
  Let $\mu$ be a probability measure on the space $\Omega$ of cookie environments which is stationary ergodic with respect to the $\Z$-shift $\theta$.
  Consider $\{0,1\}^{\Z\times\N}$, the space of arrow environments with the standard Borel $\sigma$-algebra, and let $\nu$ be the probability measure induced from $(\Omega\times U,\mu\times \PP_U)$ on $\{0,1\}^{\Z\times\N}$ by the function $F$. Then $\nu$ is stationary ergodic with respect to the shift $\theta$.
   \end{lem}

The following fact, which is straightforward, lies behind the definition of arrow environments:
\begin{fact}Given a cookie environment $\omega$, the law of a (non-random) walk moving according to the (random) arrow environment sampled from $\omega$ is the same as the quenched law of the cookie random walk on $\omega$.
\end{fact}

  Arrow environments were considered by Holmes and Salisbury in \cite{holmes2012combinatorial}. They used this construction to couple ERW on different cookie environments and deduced monotonicity results.

In this paper we consider arrow environments as a natural way to couple different processes on the same cookie environment.
The use of arrow environments gives a direct approach to distilling the ``combinatorial" part from many of the probabilistic arguments appearing in the ERW literature. See e.g. section \ref{sec:survival vs. hitting}.


\subsection{The processes $Z^+$ and $Z^-$}\label{sec:zpluszminus}

To avoid degenerate cases we introduce the following definition.
\begin{defn}
We say that sequence of arrows $b\in\{0,1\}^\N$ is \emph{non-degenerate} if there are infinitely many $i\ge 1$ for which $b(i)\neq b(i+1)$. An arrow environment $a$ is called \emph{non-degenerate} if $a(x,\cdot)$ is non-degenerate for all $x\in\Z$. The subspace of all non-degenerate arrow environments is denoted by $\mathbf{A}\subset\{0,1\}^{\Z\times\N}$
\end{defn}

Let $a\in \mathbf{A}$ be a non-degenerate arrow environment and let $y\geq 1$. We define the processes $Z^+$ and $Z^-$ (the initial value $y$ and the sequence $a$ is suppressed in the notation) as follows:
$Z^+_0=y$. Then, for every $n>0$, we define $Z^+_n$ to be the number of $1$-s until the $Z^+_{n-1}$-th zero in $a(n-1,\cdot)$. More precisely, if
\[
\Theta_n=\inf\left\{j:\sum_{i=0}^j\big[1-a(n-1,i)\big]=Z^+_{n-1}\right\},
\]
then we take $Z^+_n=\Theta_n-Z^+_{n-1}$.

We define $Z^-$ completely analogously, by replacing the roles of 0 and 1, and considering $a$ on the left half line rather than the right half line: $Z^-_0=y$, and $Z^-_n$ is the number of zeros until the $Z^-_{n-1}$-th one in $a(1-n,\cdot)$.

For ease of notation, we define for every non-degenerate $b\in\{0,1\}^\N$ the functions $U^+_b,U^-_b:\N\to\N\cup\{\infty\}$ by $U^+_b(0)=0$,
\begin{equation}\label{eq:talorenshtein}
U^+_b(x)=\inf\left\{j:\sum_{i=0}^j\big[1-b(i)\big]=x\right\}-x,
\end{equation}
and, defining $b^c$ by $b^c(x)=1-b(x)$,
\[
U^-_b(x)=U^+_{b^c}(x).
\]
Using this notation, we may simply write $Z^+_n=U^+_{a(n-1,\cdot)}(Z^+_{n-1})$ and $Z^-_n=U^-_{a(1-n,\cdot)}(Z^-_{n-1})$.

The definition given here for $Z^+$ and $Z^-$ appeared first in \cite{Amir2013excited}, where the authors of that paper considered the case of any given number of walkers on the same cookie environment and used a natural generalization of the above process.
    A slightly different version of the chains $Z^+$ and $Z^-$ was introduced and linked to one dimensional ERW by Kosygina and Zerner in 2008 \cite{kosygina2008positively}
        in the context of {\em bounded environments}, i.e.\ environments for which there is a deterministic $M$ so that $\cookieenv(x,i)=\frac{1}{2}$ for all $i>M$ and all $x\in\Z$. In such environments $Z^+$ and $Z^-$ may be viewed as a certain type of a branching process with migration. The connection of random walks to branching process with migration is traced back at least to Kesten, Kozlov, and Spitzer \cite{kesten1975limit} from 1975. An adaptation of their method to ERW was first made by Basdevant and Singh \cite{basdevant2008speed} in 2008. Using this connection, much can be said about the ERW, see e.g.  Kosygina and Zerner \cite{kosygina2008positively} and \cite{kosygina2012excited} (transience versus recurrence, ballisticity, CLT), Basdevant and Singh \cite{basdevant2008speed} and \cite{basdevant2008rate} (ballisticity and asymptotic rate of diffusivity), Peterson \cite{peterson2012large} and \cite{peterson2012strict} (law of large deviation, slow-down phenomenon, and strict monotonicity results), Rastegar and Roitershtein \cite{rastegar2011maximum} (maximum occupation time) and Dolgopyat and Kosygina \cite{dolgopyat2011central} and Kosygina and Mountford \cite{kosygina2011limit} (limit laws).

\subsection{Survival of $Z^+$ and the hitting time $T_{-1}$}\label{sec:survival vs. hitting}
    The aim of this section is to show that strict positivity of the process $Z^+$ is equivalent to the event that an excited walker on the given arrow environment never hits $-1$. This equivalence is shown to hold also for several walkers walking on the same arrow environment in \cite{Amir2013excited}.

Removing the probabilistic interpretation from the arguments of Kosygina and Zerner in Section 3 of \cite{kosygina2008positively}, a pure combinatorial condition for right-transience is obtained. Fix an arrow environment $a\in A$
and for every $m\in\Z$ set
\[
T_{m}:=\inf\{t\ge 0:X_t=m\}
\]\label{eq:defOfTminusOne}
to be the hitting time of $m$ by the excited walk $X$ on the arrow environment $a$.
Define
\[ W_0=1; \text{ } W_n=\# \{t<T_{-1}: X_t=n-1 \text{ and } X_{t+1}= n\},\text{   } n\ge 1\]
to be the total number of crossings of the edge $\{n-1,n\}$ by $X$ before hitting $-1$.

Notice that as $a$ is assumed to be in $A$, if $T_{-1}=\infty$ then $\lim_{t\to\infty} X_t = +\infty$. In particular, for each edge with non negative endpoints, the difference between the total number of its right crossings and left crossings by the walk is exactly 1. In the case where $T_{-1}<\infty$ we have an equality. Let us sum up this fact in the following remark.

\begin{rem}\label{rem:ComparingCrossings} The following hold for all $ n\ge 0$.
\begin{enumerate}
\item \label{c:1} If $T_{-1}=\infty$ then $W_n=1\ + \text{ total number of crossings of }\{n,n-1\}\text{ by } X.$  \label{eq:ComparingCrossingsTranCase}
\item \label{c:2} If $\ T_{-1}<\infty$ then $W_n = \text{ the total number of crossings of } \{n,n-1\} \text{ by } X$ before time $T_{-1}$.
\end{enumerate}
\end{rem}

\begin{lem}\label{lem:comparingZandW} For all $n\ge 0$, the following hold.
\begin{itemize}
\item if $T_{-1}<\infty$ then $Z^+_n =W_n$ \label{ZEqualsRightCrossings}
\item if $T_{-1}=\infty$ then $Z^+_n \geq W_n$ \label{ZisGraterThanRightCrossings}
\end{itemize}
Where $Z^+_n$ is defined in Section 2.2 with initial value $Z^+_0=1$.
\end{lem}

\begin{proof}
Assume first that $T_{-1}<\infty$. We will prove by induction on $n\ge 0$ that $Z^+_n = W_n$ for all $n\ge 0$. For $n=0$ we have $Z^+_0=1=W_0$ by definition. Assume now that $Z^+_n=W_n$. Since  $T_{-1}<\infty$, the last crossing of the undirected edge $\{n,n+1\}$ by $X$ before time $T_{-1}$ is a left crossing, and therefore $a(n,i)=0$, where $i$ is the total number of visits of $X$ to position $n$ before time $T_{-1}$. This implies that the number of $0$-s in $\{a(n,1),...,a(n,i)\}$ equals the total number of crossings of $\{n,n-1\}$ before time $T_{-1}$. By Remark~\ref{rem:ComparingCrossings} the last quantity equals $W_n$. Since by the induction hypothesis $Z^+_n=W_n$, We get that $Z^+_n$ is the number $0$-s in $\{a(n,1),...,a(n,i)\}$.
Now, $W_{n+1}$ is the number of ones in $\{a(n,1),...,a(n,i)\}$. The latter is exactly the number of $1$-s prior to $Z^+_n$ $0$-s in $a(n,\cdot)$, which is defined to be $Z_{n+1}^+$.

Consider now the case $T_{-1}=\infty$. Again, we will prove by induction on $n\ge 0$ that $Z^+_n \geq W_n$. For $n=0$ we have $Z_0=1=W_0$. Assume by induction that $Z^+_n\ge W_n$. Let $i$ be the total number of visits of $X$ to place $n$. Note that as $a\in A$ the process $X$ is transient and so $i<\infty$ and $a(n,i)=1$. By Lemma \ref{rem:ComparingCrossings} $W_n$ equals the number of $0$'s in $\{a(n,1),...,a(n,i)\}$ plus $1$. As in the first case by the induction hypothesis
$Z^+_n$ is greater than or equal to the number of $0$'s in $\{a(n,1),...,a(n,i)\}$ plus $1$. Now, $W_{n+1}$ is the number of ones in $\{a(n,1),...,a(n,i)\}$. The latter is exactly the number of $1$-s prior to the $(Z_n^+ -1)$-st zero in $a(n,\cdot)$, and this number is less than or equal to $Z^+_{n+1}$.
\end{proof}

As a result, we get the following theorem.
\begin{thm}\label{thm:SurvivalEquivNonReturning}
$T_{-1}<\infty$ if and only if $Z^+_n=0$ for some $n$, and $T_{1}<\infty$ if and only if $Z^-_n=0$ for some $n$.
\end{thm}
\begin{proof}
For the first equivalence, assume first that $T_{-1}<\infty$, then $M:=\max\{X_t:t<T_{-1}\}<\infty$. Therefore by Lemma \ref{ZEqualsRightCrossings} we get that $Z^+_{M+1}=W_{M+1}=0$.
On the other hand, if $T_{-1}=\infty$, then $W_n\ge 1$ for all $n\ge 0$. Now if $Z^+_n=0$ then by Lemma \ref{ZisGraterThanRightCrossings} we have that $W_n\le Z^+_n =0$, a contradiction.
Considering $\bar{a}$ instead of $a$, where $\bar a (n,i)= 1- a (-n,i)$, the argument from the beginning of the present subsection shows
the second equivalence.
\end{proof}

\begin{rem}
The proof of Theorem \ref{thm:SurvivalEquivNonReturning} shows that if $M:=\max\{X_t:t<T_{-1}\}\in\N\cup\{\infty\}$ is the maximal position of the walk before reaching $-1$ and $\tau=\inf \{ n\ge 0 : Z^+_n = 0\} \in\N\cup\{\infty\}$ is the extinction time of $Z^+$, then $\tau = M+1$, where, by convention, $\infty = \infty + 1$.
\end{rem}

\subsection{Subduality of $Z^+$ and $Z^-$}\label{sec:comoz}

    There are two immediate properties of $U^+$ and $U^-$ which will be crucial for our arguments:
    \begin{obs}\label{obs:propsOfU} The following hold for any $b\in\{0,1\}^\N$.
    \begin{enumerate}
    \item\label{itm:UNondec} $U^+_b(x)$ and $U^-_b(x)$ are nondecreasing in $x$.
    \item\label{itm:UplusUminusSmallerThenIdentity} $U^+_b \circ U^-_b(x) <  x$ for all $x\in\Z_+$ (and equivalently $U^-_b \circ U^+_b(x) <  x$).
    \end{enumerate}
    \end{obs}

    The next lemma is simple but crucial for the proof of Theorem \ref{thm:directionalZeroOneLaw}. For $z\in\Z$ we define $\theta^z:A\to A$ by $(\theta^z a)(x,i)=a(x+z,i)$, $x\in\Z$ and $i\in\N$ is the $\Z$-shift map by $z$ steps to the left.

    \begin{lem}[Subduality]\label{lem:subduality}
    Assume that for the arrow environment $a\in A$ the process $Z^+$ with initial value $Z^+_0=x$ has $Z_l^+\ge y$. Then on the shifted arrow environment $\theta^{l-1}a$, the process $Z^-$ with initial value $Z^-_0=y$ has $Z_l^-\le x$.
    \end{lem}

    \begin{proof}
    Property \ref{itm:UplusUminusSmallerThenIdentity} of Observation \ref{obs:propsOfU} gives us that $U^-_{a({l-1},\cdot)} \circ U^+_{a({l-1},\cdot)}(x)\le x$ for all $x\in\Z_+$. Using this together with the monotonicity property \ref{itm:UplusUminusSmallerThenIdentity} of Observation \ref{obs:propsOfU} $l$ times, we get:
    $$ U^-_{a(0,\cdot)} \circ \ldots \circ U^-_{a({l-1},\cdot)} \circ U^+_{a({l-1},\cdot)} \circ \ldots \circ U^+_{a(0,\cdot)}(x)\le x.$$
    In other words $U^-_{a(0,\cdot)} \circ \ldots \circ U^-_{a({l-1},\cdot)}(m)\le x$, where $m:=Z_l^+=U^+_{a({l-1},\cdot)} \circ \ldots \circ U^+_{a(0,\cdot)}(x)$. By the assumption $m\ge y$ and so by the monotonicity property \ref{itm:UplusUminusSmallerThenIdentity} of Observation \ref{obs:propsOfU} also $U^-_{a(0,\cdot)} \circ \ldots \circ U^-_{a({l-1},\cdot)}(y)\le x$.
    To finish, note that the left hand side of the last inequality is by definition $Z_l^-\le x$, where the process $Z^-$ is defined on the shifted environment $\theta^{l-1}a$ with initial value $Z^-_0=y$.
    \end{proof}

   \section{A Kalikow type 0-1 law}\label{sec:0-1 Law for Directional Transience}

The main purpose of this section is to present the proof, originally by Kosygina and Zerner, of a Kalikow type 0-1 law, and by it set the ground for the proof of our main result, which is to be found in the next section. Remember the events $A^+$ and $A^-$ from the introduction.
    \begin{thm}\cite[Theorem 3.2]{kosygina2012excited} \label{thm:KalikowType}
    Let $\mu$ be a stationary ergodic and elliptic probability measure over cookie environments. Then $\PP_0(A^+\cup A^-)\in\{0,1\}$.
    \end{thm}


An interesting question 
related to Theorem \ref{thm:KalikowType} is the following:

\begin{prob}
Can the ellipticity assumption be weakened in Theorem \ref{thm:KalikowType}? If so, can it be weakened also in Theorem \ref{thm:directionalZeroOneLaw}?
\end{prob}

    As a convention, we say a probability measure $\mu$ over cookie environments has a given property of cookie environments if every cookie environment has it $\mu$-a.s.\ We say $\mu$ satisfies a given property $P$ of arrow environments if almost every arrow environment has the property $P$ with respect to the annealed measure $\annealedP$ associated to $\mu$. For example $\mu$ is elliptic if $\mu ((0,1)^{\Z\times\N})=1$, and $\mu$ is non-degenerate if the induced measure on arrow environments satisfies $\annealedP (a\in A)=1$.

    Unless otherwise mentioned, from now onwards we assume that $\mu$ is a stationary ergodic and elliptic probability distribution over cookie environments. As a first reduction to the proofs of both Theorem \ref{thm:KalikowType} and Theorem \ref{thm:directionalZeroOneLaw} we will use the following definition and results from Section 2 of Kosygina and Zerner \cite{kosygina2012excited}, regarding the question of finiteness of the ERW range.
    For $x\in\Z$ and $\cookieenv\in\Om$, let $R(x,\cookieenv)$ be the event that $\sum_{i=1}^\infty (\cookieenv(x,i))<\infty$ and $L(x,\cookieenv)$ the event that $\sum_{i=1}^\infty (1-\cookieenv(x,i))<\infty$.

    Notice that it follows from the Borel-Cantelli Lemma that a probability measure $\mu$ over cookie environments satisfies $\mu(R(x,\cookieenv))=\mu(L(x,\cookieenv))=0$ for every $x\in\Z$, if and only if it is non-degenerate. Moreover, the Borel-Cantelli Lemma also implies the following lemma (see e.g. \cite[Lemma 2.2]{kosygina2012excited}).

    \begin{lem}\label{lem:NotGettingStuck}
    Assume that $\mu$ is a non-degenerate probability measure over cookie environments. Then $\annealedP_0$-a.s.\ $$\limsup_{n\to\infty}X_n\in\{-\infty,+\infty\}\ \mbox{and}\ \liminf_{n\to\infty} X_n\in\{-\infty,+\infty\}$$
    \end{lem}
    \begin{thm}\label{thm:range}\cite[Theorem 2.3]{kosygina2012excited}
    Let $\mu$ be a stationary ergodic and elliptic probability measure over cookie environments.
    \begin{enumerate}
    \item[(a)] If $\mu(R(0,\cookieenv)))>0$ and $\mu(L(0,\cookieenv)))>0$ then
    the range is $\annealedP_0$-a.s.\
    finite.
    \item[(b)] If $\mu(R(0,\cookieenv)))=0$ and $\mu(L(0,\cookieenv)))>0$ then $\annealedP_0(A^+)=1$.
    \item[(c)] If $\mu(R(0,\cookieenv)))>0$ and $\mu(L(0,\cookieenv)))=0$ then $\annealedP_0(A^-)=1$.
    \item[(d)] If $\mu(R(0,\cookieenv)))=0$ and $\mu(L(0,\cookieenv)))=0$ then the range is $\annealedP_0$-a.s.\
    infinite.
    \end{enumerate}
    \end{thm}
    The proof of Theorem \ref{thm:range} is omitted.
    \begin{cor}\label{cor:ellReduction}
    Let $\mu$ be a stationary ergodic and elliptic probability measure over cookie environments. If $\mu$ is not non-degenerate then $\annealedP_0(A^+),\annealedP_0(A^-) \in \{0,1\}$. In particular, in this case $\annealedP_0(A^+\cup A^-) \in \{0,1\}$.
    \end{cor}

    Given a walk $X_n$ on $\Z$ , a \emph{right excursion} (from $0$) is a sequence of steps $X_{\tau_0},\ldots, X_{\tau_1}\le \infty$ of the walk such that $X_{\tau_0}=0$, either $X_{\tau_1}=0$ or $\tau_1=\infty$, and $X_t>0$ for all $\tau_0<t<\tau_1$.
     Call $m\geq 0$ an \emph{optional regeneration position} for an arrow environment $a$ if the walk started at $m$ never hits $m-1$, that is if $T_{m-1}=\infty$.
     We call $m\geq 0$ a \emph{regeneration position} if in addition, when starting the walk from $0$, the walk $X_n$ reaches $m$ after some finite time.
     Note that the arrow environment on $[m,\infty)$ remains unchanged until the walker reaches $m$ for the first time (if it ever does). It follows that if $m$ is an (optional) regeneration position, and the walk $X_n$ reaches $m$, then afterwards it will never return to $m-1$.

     \begin{lem}\label{le:inf_opt_reg} Let $\mu$ be a stationary ergodic probability measure over cookie environments.
     If $\mathbb{P}_0(T_{-1}=\infty)>0$ then there are $\mathbb{P}_0$-a.s.\ infinitely many optional regeneration positions.
     \end{lem}
     \begin{proof}
       Let $p:=\mathbb{P}_0(T_{-1}=\infty)>0$.
      By stationarity of the arrow environment, $\mathbb{P}_m(T_{m-1}= \infty) = p$ for any $m\geq 0$. By Lemma \ref{lem:arrow_ergod} we may apply the ergodic theorem to the measure $\nu$ to get
    $$
    \frac{1}{n}\sum_{m=1}^n \mathbf{1}_{\{m \text{ is a optional regeneration position}\}} \to p\,\, \nu-\text{a.s.}.
    $$
    In particular there are $\nu$-a.s.\ infinitely many optional regeneration positions.
     \end{proof}

    The following lemma is a part of Lemma 8 of \cite{kosygina2008positively}, which is proved for the the i.i.d.\ case.

    \begin{lem}\label{lem:TisInfiniteFinitelyManyExcursions} Let $\mu$ be a stationary ergodic and non-degenerate probability measure over cookie environments. If $\mathbb{P}_0(T_{-1}=\infty)>0$ then there are a.s.\ only finitely many right excursions.
    \end{lem}

    \begin{proof}
    Consider $\limsup X_n$. If $\limsup X_n<\infty$, then by Lemma \ref{lem:NotGettingStuck} $A^-$ holds. In particular, the number of right excursions is finite. If $\limsup X_n=\infty$, then since by Lemma \ref{le:inf_opt_reg} on the event $\limsup X_n=\infty$ there a.s.\ exist (infinitely many) optional regeneration positions, the walk hits such a position $m$ at some finite time and from that time on it never returns to $m-1$, let alone $0$, and thus there are only finitely many right excursions.
    \end{proof}

    Also the following lemma is a part of Lemma 8 of \cite{kosygina2008positively}.

    \begin{lem}\label{lem:TisFiniteExcursionsAreFinite}
    If $\PP_{\cookieenv,0}(T_{-1}<\infty)=1$ and $\cookieenv$ is elliptic, that is $\cookieenv\in(0,1)^{\Z\times\N}$, then all right excursions are $\PP_{\cookieenv,0}$-a.s.\ finite.
    \end{lem}
    \begin{proof}
    The proof uses a finite modification argument which is standard (see \cite{kosygina2008positively} proof of Lemma 8, or \cite{kosygina2012excited} (3.2) and Figure 1.\ there).
    For convenience we shall supply a sketch. We will prove that the $i$-th right excursion is a.s.\ finite by induction on $i$. For $i=0$ this is trivial.
    Assume now that the first $i$ right excursions are a.s.\ finite and consider the past including the first step of the $(i+1)$-st excursion.
    The event that the last excursion is finite depends only on what the walk has done in places $x>0$. Therefore, the probability that the $(i+1)$-st excursion is finite given that the past does not change when we modify parts of the past as long as we do not change the parts when $x>0$.
    In particular it remains the same when we erase all visits to the negative integers and visits to zero are concatenated in time (simply by replacing enough of the first arrows above $0$ to be right arrows).
     As the modified event has positive probability, conditioning on it, the probability for finiteness of the $(i+1)$-st excursion equals the probability that the first excursion is finite conditioned on making a pre-given sequence of first steps on the positive half line.
     This equals $1$ by the assumption of the lemma since by ellipticity of $\cookieenv$ there is a positive probability to make any pre-given finite sequence of moves.
    \end{proof}

    \begin{cor}\label{cor:TransVsExtinctions}[\cite[Lemma 3.3]{kosygina2012excited}, \cite[Corollary 3.7]{Amir2013excited}] Let $\mu$ be an elliptic and non-degenerate probability measure over cookie environments.
    $\PP_0(T_{-1}=\infty ) > 0$ if and only if $\PP_0(A^+) > 0$
    \end{cor}
    \begin{proof}
    Note that if $T_{-1}=\infty$, then by Lemma \ref{lem:NotGettingStuck} $\liminf X_n=+\infty$ which yields $A^+$.
    For the other implication, assume $\PP_0(T_{-1}<\infty ) = 1$, then $\PP_{\cookieenv,0}(T_{-1}<\infty ) = 1$ for $\mu$-a.e.\ $\cookieenv\in \Om$. By Lemma \ref{lem:TisFiniteExcursionsAreFinite} $\PP_{\cookieenv,0}$-a.s.\ all right excursions are finite and in particular $\PP_{\cookieenv,o}$-a.s.\ $X_n \nrightarrow +\infty$ for $\mu$-a.e.\ $\cookieenv\in \Om$. In other words, in this case $\PP_0(A^+) = 0$.
    \end{proof}

    \begin{proof}[Proof of Theorem \ref{thm:KalikowType}] First note that in the case where $\mu$ is not non-degenerate, the conclusion of the theorem follows from Corollary \ref{cor:ellReduction}. For the rest of the proof we assume that $\mu$ is also non-degenerate.
    If $\PP_0 (T_{-1} = \infty) > 0$ (resp.\ $\PP_0 (T_{1} = \infty) > 0$) then by Lemma \ref{lem:TisInfiniteFinitelyManyExcursions} there are $\PP_0$-a.s.\ only finitely many right (resp.\ left) excursions.
    In particular $\PP_0$-a.s.\ the walk visits $0$ only finitely many times from the right (resp.\ left).
      By the assumption, for every $m$
      \[ \PP_0\left[\prod_{i=m}^\infty (1-\cookieenv(0,i))=0\right]=1 \ \ \left(\mbox{resp. } \PP_0\left[\prod_{i=m}^\infty \cookieenv(0,i)=0\right]=1\right) \]
      so we have $\PP_0$-a.s.\ only finitely many visits to zero, which implies the occurrence of the event $A^+\cup A^-$.

    On the other hand,
    if $\mathbb{P}_0 (T_{-1}<\infty)=1$ and $\mathbb{P}_0 (T_{1}<\infty)=1$ then by Lemma \ref{lem:TisFiniteExcursionsAreFinite} the walk is $\PP_0$-a.s.\ not transient to the right and not transient to the left. This means it is $\PP_0$-a.s. recurrent.
   \end{proof}

    For $y,n\in\Z_+$ and $B\subset\Z_+$ denote by $\PP^y(Z_n^+ \in B)$ the probability that the process $Z^+$ with initial value $y$ satisfies $Z_n^+\in B$. $\PP^y(Z_n^-\in B)$ is defined similarly.

    Let $S_+$ and $S_-$ be the events that $\{Z_n^+ >0 \text{ for all } n\}$ and $\{Z_n^- >0 \text{ for all } n\}$, respectively.

    \begin{cor}\label{cor:Reduction}
      Let $\mu$ be a stationary ergodic, elliptic and non-degenerate probability measure over cookie environments. $\PP_0(A^+)=1$ if and only if $\PP^1(S_+)>0$ and $\PP^1(S_-)=0$.
    \end{cor}
    \begin{proof}
    If $\PP_0(A^+)=1$ then by
    Corollary \ref{cor:TransVsExtinctions} $\PP_0(T_{-1}=\infty ) > 0$ and by Theorem \ref{thm:SurvivalEquivNonReturning} also $\PP^1(S_+)>0$.
    Assume for contradiction $\PP^1(S_-)>0$, then by Theorem \ref{thm:SurvivalEquivNonReturning} also $\PP_0(T_{1}=\infty ) > 0$,
    and so by Corollary \ref{cor:TransVsExtinctions} also $\PP_0(A^-)>0$. This contradicts the assumption.

    For the other direction, again by Theorem \ref{thm:SurvivalEquivNonReturning} and Corollary \ref{cor:TransVsExtinctions} we know that $\PP_0(A^+)>0$ and $\PP_0(A^-)=0$.
    By Theorem \ref{thm:KalikowType}, we have $\PP_0(A^+\cup A^-)=1$, and so $\PP_0(A^+)=1$.
    \end{proof}

    \section{Proof of the main result}\label{sec:pfmain}
    This section is devoted to the proof of Theorem \ref{thm:directionalZeroOneLaw}. The following is a key proposition.
    \begin{prop}\label{prop:main}
    Assume that $\mu$ is a stationary ergodic and elliptic probability measure over cookie environments. If $\PP^1(S_+)>0$, then $\PP^1(S_-)=0$.
    \end{prop}

    We shall first prove Theorem \ref{thm:directionalZeroOneLaw} assuming Proposition \ref{prop:main}, and then turn to proving Proposition \ref{prop:main}.

    \begin{proof}[Proof of Theorem \ref{thm:directionalZeroOneLaw}]
    By Corollary \ref{cor:ellReduction} we may assume that $\mu$ is also non-degenerate. If $\PP^1(S_+)>0$ then by Proposition \ref{prop:main}
    $\PP^1(S_-)=0$ and therefore by Corollary \ref{cor:Reduction} $\PP_0(A^+)=1$. Symmetrically, if $\PP^1(S_-)>0$ then $\PP_0(A^-)=1$. To deal with the last case, namely that $\PP^1(S_-\cup S_+)=0$, note that  Corollary \ref{cor:TransVsExtinctions} implies that $\PP_0(A^+\cup A^-)=0$. Since $\mu$ is non-degenerate, then by Lemma \ref{lem:NotGettingStuck} it holds that $\PP_0(X_n=0\text{ i.o.})=1$. This completes the proof of the theorem.
    \end{proof}

    \begin{rem}\label{rem:easyCases}
  In some specific cases, e.g. when $\mu$ is uniformly elliptic and stationary ergodic or when it is i.i.d.\ and elliptic, there are case specific proofs of Proposition \ref{prop:main}
  which are significantly simpler than the one provided below for the general case.
  For example, in the i.i.d. case both $Z^+$ and $Z^-$ are Markov chains, and thus it suffices to show that there exists $M$ such that  $\{\PP^k(\exists n\, Z^-_n<M):k=1,2,3,\ldots\}$ is bounded away from zero. This, however, follows
  immediately from Lemma \ref{lem:getting close to one} and subduality (Lemma \ref{lem:subduality}).
  If, on the other hand, we assume uniform ellipticity (but no longer i.i.d.), instead of the density statement of Lemma \ref{lem:upperdens0} one can easily show that $Z_n^+$ goes to infinity,
  which in turn allows us to immediately use Lemma \ref{lem:subduality} to finnish the proof.

    \end{rem}

   We now prove Proposition \ref{prop:main}. We shall divide the proof into several steps. To the end of the paper we assume that all the assumptions of Proposition \ref{prop:main} hold.

    \begin{lem}\label{lem:getting close to one}
    For every $\epsilon>0$ there is some $y\in\Z_+$ so that $\PP^y[S_+]>1-\epsilon$.
    \end{lem}

    \begin{proof}
    $Z_n^+>0$ for all $n\ge 0$ if and only if $T_{-1}=\infty$, and so as in Lemma \ref{le:inf_opt_reg}, by stationarity and ergodicity of the environment
    there are a.s.\ infinitely many optional regeneration positions.
    In particular there is an a.s. finite (random) first optional regeneration position $\phi>0$. Consider now the process $Z^+$ at time $\phi$. If $Z^+$ is positive at time $\phi$ then it will stay positive forever. Fix $\epsilon>0$. Let $m$ be a large enough number so that $\PP_0(\phi> m)<\frac{\epsilon}{2}$. Set $k_m=1$ and sequentially choose sufficiently large $k_{m-1},\ldots,k_0\in\N$ so that $\PP(U_{a(j-1,\cdot)}(k_{j-1})< k_j)<\frac{\epsilon}{2m}$ for all $1\leq j\leq m$. Setting $y=k_0$, we have $$\PP^y[S_+]\ge \PP^y[Z^+_\phi\ge 1]\ge \PP ( U_{a(i-1,\cdot)}(k_{i-1}) \ge k_{i},i=1,...,m,\phi\leq m ) \ge 1-\epsilon $$ by union bound.
    \end{proof}

    For a set $B\subset\Z_+$ denote by $\dens(B)$ the upper density of $B$, that is $$\dens(B)=\limsup_{n\to\infty}\frac{\#\{j\in B:j<n\}}{n}.$$
    \begin{lem}\label{lem:upperdens0}
    For every initial $y\ge 1$ it holds that $\PP^y\big[\dens(\{n:Z^+_n<k\})=0\ |\ S_+\big]=1$ for every $k\ge 0$.
    \end{lem}
    \begin{proof}
    Fix $k>0$. For $\gamma>0$ and $x\in\Z$ let $A_{\gamma,x}$ be the event that $\prod_{i=1}^k(1-\cookieenv(x,i))>\gamma$. By stationarity $g(\gamma):=\mu[A_{\gamma,x}]$ is independent of $x$. By ellipticity, $\lim_{\gamma\to 0}g(\gamma)=1$. By ergodicity, the set $A_{\gamma}=\{x: A_{\gamma,x}\}$ has density $g(\gamma)$.  Let $r$ be a natural number, let $B_r$ be the event that $\dens(\{n:Z^+_n<k\})>\frac{1}{r}$ and let $\gamma$ be small enough so that $g(\gamma)+\frac{1}{r}>1$. Then, on $B_r$
    there are infinitely many $n$ such that both events $Z_n^+\le k$ and $A_{\gamma,n+1}$ occur.
       Set $\mathcal{F}_n=\sigma \{\omega, a(0,\cdot),...,a(n-1,\cdot)\}$ be the $\sigma$-algebra generated by the all the cookies and the first $n$ piles of arrows to the right of and including $0$ and let $M_n=\PP\big[S_+^c\ |\ \mathcal{F}_n \big]$ where $S_+^c$ is the complement of $S_+$. Then $(M_n)_{n\ge 1}$ is a bounded martingale with respect to the filtration $(\mathcal{F}_n)_{n\ge 1}$ converging $\PP$-a.s.\ to $\one_{S_+^c}$. Now, the occurrence of $B_r$ implies that there are infinitely many $n$ for which both events $Z_n^+\le k$ and $A_{\gamma,n+1}$ occur, and therefore there are infinitely many $n$ with $M_n \ge \gamma$. In particular, on $B_r$, $\one_{S_+^c}>\gamma$, implying that $S_+^c$ occurs, so $\PP\big[B_r\ |\ S_+\big]=0$. Since $r$ was chosen arbitrarily, we are done.
    \end{proof}

    \begin{lem}\label{lem:LargeProbPlusGetsAnyHeight}
    Let $\epsilon>0$ and let $y$ be from Lemma \ref{lem:getting close to one}. For every $k\ge 0$ there is some $l$ so that $\PP^y[Z_l^+>k]>1-2\epsilon$.
    \end{lem}
    \begin{proof}
    Assume not, then by Lemma \ref{lem:getting close to one} there are $k$ and $\epsilon$ so that $\PP^y\big[Z_l^+>k\ |\ S_+\big]\le 1-\epsilon=:\lambda<1$ for all $l$.
    By linearity of expectation, $\E^y\big[\#\{l<n:Z^+_l>k\}\ |\ S_+\big]\le n \lambda$. Fix some $\lambda<\delta<1$, then by the Markov inequality we have $$\PP^y\big[\frac{\#\{l<n:Z^+_l>k\}}{n}>\delta\ |\ S_+ \big]\le \frac{\lambda}{\delta}.$$
    In other words,
    $$\PP^y\big[\frac{\#\{l<n:Z^+_l\le k\}}{n}\ge 1-\delta\ |\ S_+ \big]\ge 1- \frac{\lambda}{\delta}=:\alpha>0. $$
    But therefore
    $$\PP^y\big[\text{ there are infinitely many $n$ such that }\frac{\#\{l<n:Z^+_l\le k\}}{n}\ge 1-\delta\ | \ S_+ \big]\ge \alpha,$$
    contradicting Lemma \ref{lem:upperdens0}.
    \end{proof}

    By translation invariance of the probability measure $\mu$ we get from the Subduality Lemma \ref{lem:subduality} the corollary below.
    Denote by $\PP^k_{r}[Z_l^-\le y]$, $r\in\Z$, the probability that on the $r$-shifted arrow environment $\theta^r a$, the process $Z^-$ with initial value $k$ satisfies $Z_l^-\le y$.
    \begin{cor}\label{cor:ShiftInvarianceOfTimeTrick}
    For every $k\in\N$, $r_1,r_2\in\Z$ and $\epsilon>0$ there is some $l\in\N$ so that $\PP^k_{r_1}[Z_l^-\le y] \ge \PP^y_{r_2}[Z_l^+>k]\ge 1-2\epsilon$, where $y$ is as in Lemma \ref{lem:getting close to one}.
    \end{cor}
    \begin{proof}
    Fix $k\in\N$ and $\epsilon>0$ and let $l$ be the one guaranteed in Lemma \ref{lem:LargeProbPlusGetsAnyHeight}. $r_1,r_2\in\Z$, then the right inequality follows from stationarity of $\mu$, and the left inequality follows from the Subduality Lemma \ref{lem:subduality} and stationarity of $\mu$.
    \end{proof}

    \begin{lem}\label{lem:a lot of chances to die}
    For every $\epsilon>0$ there are $n_1<n_2<\ldots$ so that $\PP^1[Z^-_{n_i} > y]<3\epsilon$, where $y$ is as in Lemma \ref{lem:getting close to one}.
    \end{lem}
    \begin{proof}
    Fix $m_1=0$. There is $k_1$ so that $\PP^1[Z^-_{m_1}>k_1]<\epsilon$. Let $l_1$ be the $l$ guaranteed by Corollary \ref{cor:ShiftInvarianceOfTimeTrick} for $k=k_1$.
    Define $n_1=m_1+l_1$, then by Corollary \ref{cor:ShiftInvarianceOfTimeTrick} $\PP^1[Z^-_{n_1}<y]\ge\PP^1[Z^-_{m_1}\le k_1,Z^-_{n_1}<y]\ge 1-3\epsilon$.
    Let $m_2>n_1$. There is $k_2$ so that $\PP^1[Z^-_{m_2}>k]<\epsilon$ Let $l_2$ be the $l$ guaranteed by Corollary \ref{cor:ShiftInvarianceOfTimeTrick} for $k=k_1$ and $r_1=m_2$. Define $n_2=m_2+l_2$, then $\PP^1[Z^-_{n_2}<y]\ge\PP^1[Z^-_{m_2}\le k,Z^-_{n_2}<y]\ge 1-3\epsilon$.
    Assume that $n_1<...<n_r$ were chosen so that $\PP^1[Z^-_{n_i}<y]\ge 1-3\epsilon$ for all $1\le i\le r$.
    At the $(r+1)$-st step, fix $m_{r+1}>n_r$. There is $k_{r+1}$ so that $\PP^1[Z^{-}_{m_{r+1}} > k_{r+1}] <\epsilon$. Let $l_{r+1}$ be the $l$ guaranteed by Lemma \ref{lem:LargeProbPlusGetsAnyHeight} for $k=k_{r+1}$ and $r_1=m_{r+1}$. Define $n_{r+1}=m_{r+1}+l_{r+1}$, then \[\PP^1[Z^{-}_{n_{r+1}}<y]\ge\PP^1[Z^{-}_{m_{r+1}}\le k_{r+1},Z^-_{n_{r+1}}<y]\ge 1-3\epsilon.\]
    \end{proof}

    \begin{proof}[Proof of Proposition \ref{prop:main}]
    Assume that $\PP^1\big[S_+\big]>0$. Let $\delta>0$. We will show that $\PP^1\big[Z^-_{n} > 0 \text{ for all } n\big]\le\delta$. Let $\epsilon=\frac{\delta}{4}>0$.
    By Lemma \ref{lem:a lot of chances to die}, there are $n_1<n_2<\ldots$ so that $\PP^1\big[Z^-_{n_i} \le y\big]\ge 1-3\epsilon$, where $y$ is as in Lemma \ref{lem:getting close to one}. As in the proof of Lemma \ref{lem:upperdens0}, set
    \[
    A_{\gamma,x}=\left\{\prod_{i=1}^y \cookieenv(x,i)>\gamma\right\} \mbox{ for }\gamma>0\mbox{ and }x\in\Z.
    \]
    By stationarity $g(\gamma):=\mu[A_{\gamma,x}]$ is independent of $x$. By ellipticity, $\lim_{\gamma\to 0}g(\gamma)=1$.
    Let $\gamma>0$ be small enough so that $g(\gamma)>1-\epsilon$. Then $\PP^1[Z^-_{n_i} \le y,A_{\gamma,-n_i}]\ge 1-4\epsilon=1-\delta$ for all $i\ge 1$.
    Let $D$ be the event  that there are infinitely many $i$ such that $Z^-_{n_i} \le y \text{ and }A_{\gamma,-n_i}$. Then $\PP^1[D]\ge 1-\delta$.

    Set $\mathcal{F}_n=\sigma \{ \cookieenv, a(0,\cdot),...,a(-(n-1) \}$ be the $\sigma$-algebra generated by all the cookies and the first $n$ piles of arrows to the left of and including $0$ and let $M_n=\PP\big[S_+^c\ |\ \mathcal{F}_n \big]$. Then $(M_n)_{nge 1}$ is a bounded martingale with respect to the filtration $(\mathcal{F}_n)_{n\ge 1}$ converging a.s.\ to $\one_{S_-^c}$, and so if the event $D$ occurs, then also does $S_-^c$.
    Therefore
    \[
    \PP^1[Z^-_{n} > 0 \text{ for all } n]\le 1-\PP^1[D]\leq\delta.
    \]
    Since $\delta$ was arbitrary, we are done.
    \end{proof}

\section*{Acknowledgments}
We thank Itai~Benjamini, Xiaoqin~Guo, Gady~Kozma, Igor~Shinkar and Ofer~Zeitouni for useful discussions. We also thank Jonathon Peterson
for reading an earlier version of the manuscript. We thank the anonymous referees for valuable comments including the open problem regarding ellipticity that helped us to improve the mathematical results and the style of the paper.
The research of N.B. and T.O. was partially supported by ERC StG grant 239990. The research of G.A. was supported by Israeli Science Foundation grant ISF 1471/11.

    \bibliography{ProbabilityBib}{}

\begin{thebibliography}{10}

\bibitem{Amir2013excited}
G.~Amir and T.~Orenshtein.
\newblock Excited mob.
\newblock {\em arXiv preprint arXiv:1307.6052}, 2013.

\bibitem{basdevant2008speed}
A.L. Basdevant and A.~Singh.
\newblock {On the speed of a cookie random walk}.
\newblock {\em Probability Theory and Related Fields}, 141(3):625--645, 2008.

\bibitem{basdevant2008rate}
A.L. Basdevant and A.~Singh.
\newblock {Rate of growth of a transient cookie random walk}.
\newblock {\em Electron. J. Probab}, 13:811--851, 2008.

\bibitem{benjamini2003excited}
I.~Benjamini and D.B. Wilson.
\newblock {Excited random walk}.
\newblock {\em Electron. Comm. Probab}, 8(9):86--92, 2003.

\bibitem{dolgopyat2011central}
D.~Dolgopyat.
\newblock Central limit theorem for excited random walk in the recurrent
  regime.
\newblock {\em ALEA Lat. Am. J. Probab. Math. Stat}, 8:259--268, 2011.

\bibitem{holmes2012combinatorial}
M.~Holmes and T.~S. Salisbury.
\newblock A combinatorial result with applications to self-interacting random
  walks.
\newblock {\em J. Comb. Theory Ser. A}, 119(2):460--475, 2012.

\bibitem{kesten1975limit}
H.~Kesten, MV~Kozlov, and F.~Spitzer.
\newblock {A limit law for random walk in a random environment}.
\newblock {\em Compositio Math}, 30:145--168, 1975.

\bibitem{kosygina2011limit}
E.~Kosygina and T.~Mountford.
\newblock Limit laws of transient excited random walks on integers.
\newblock {\em Annales de l'Institut Henri Poincar{\'e}, Probabilit{\'e}s et
  Statistiques}, 47(2):575--600, 2011.

\bibitem{kosygina2008positively}
E.~Kosygina and M.P.W. Zerner.
\newblock {Positively and negatively excited random walks on integers, with
  branching processes}.
\newblock {\em Electron. J. Probab}, 13:1952--1979, 2008.

\bibitem{kosygina2012excited}
E.~Kosygina and M.P.W. Zerner.
\newblock Excited random walks: results, methods, open problems.
\newblock {\em Bulletin of the Institute of Mathematics Academia Sinica (New
  Series)}, 8:105--157, 2013.

\bibitem{peterson2012large}
J.~Peterson.
\newblock Large deviations and slowdown asymptotics for one-dimensional excited
  random walks.
\newblock {\em Electron. J. Probab}, 17(48):1--24, 2012.

\bibitem{peterson2012strict}
J.~Peterson.
\newblock Strict monotonicity properties in one-dimensional excited random
  walks.
\newblock {\em Markov Processes and Related Fields}, 19(4):721--734, 2013.

\bibitem{rastegar2011maximum}
R.~Rastegar and A.~Roitershtein.
\newblock Maximum occupation time of a transient excited random walk on
  $\mathbb{Z}$.
\newblock {\em arXiv preprint arXiv:1111.1254}, 2011.

\bibitem{zerner2005multi}
M.P.W. Zerner.
\newblock {Multi-excited random walks on integers}.
\newblock {\em Probability Theory and Related Fields}, 133(1):98--122, 2005.

\end{thebibliography}
    \bibliographystyle{plain}

\begin{appendix}
    \section{Proof sketch of Lemma \ref{lem:arrow_ergod}}
    \begin{proof}[\nopunct]
    The function $F$ is a measurable function from the product space $\Omega\times U$ with the shift $\theta\times\theta$ to the space $\mathbf{A}$ with the shift $\theta$, so that the measure on the latter is obtained from the former by $F$. It is straightforward to verify that $\mu\times \PP'$ is stationary with respect to $\theta\times\theta$.
    Note that
    \begin{equation}\label{eq:ShiftComutesWithF}
    F^{-1}[\theta A]= (\theta\times\theta) F^{-1}[A] \text{ for every }A\subset \mathbf{A}.
    \end{equation}
    Hence the stationarity of $\nu$ follows from the stationarity of $\mu\times\PP'$.

    For the proof of ergodicity, first note that it is enough to show that $( \Omega\times U , \mu \times \PP', \mathcal{B}_{\Omega} \times \mathcal{B}_{U},\theta\times\theta)$ is ergodic. Indeed, using \eqref{eq:ShiftComutesWithF} the inverse image under $F$ of each $\theta$-invariant set is $(\theta\times\theta$)-invariant, and so it must be of either $\PP\times P$-measure $0$ or $1$.
    To prove the ergodicity of $\mu\times \PP'$, let $f:\Omega\times U\to [0,1]$ be a $\mu\times \PP'$-measurable function. We shall show that $f$ is a constant function.
    Denote by $E$ the expectation operator with respect to $\mu\times \PP'$.
    First note that $\varphi:=E(f|\mathcal{B}_\Omega)$ is a $\theta$ invariant function on $\Omega$ and so by ergodicity it is $\mu$-a.s. constant in $[0,1]$.

    Let $f_n=E[f|\mathcal{B}_\Omega \times \sigma \left( u(-n,\cdot),...,u(n,\cdot)\right)]$. Then $E[|f-f_n|]\to 0$ as $n\to \infty$, where \\ $\sigma \left( u(-n,\cdot),...,u(n,\cdot) \right)\subset\mathcal{B}_U$ is the minimal sub $\sigma$-algebra containing the $\Z$-coordinates $-n,...,n$.
    Let $\epsilon>0$ and let $n_0$ be large enough so that for all $n\ge n_0$ $E[|f-f_n]<\epsilon$.
    Let $\tilde{f_n}=(\theta\times\theta)^{3n}f_n$ be the $3n$ steps left shift of $f_n$.
    Note that, since $\PP'$ is the product measure,
   $f_n$ and $\tilde f_n$ are independent conditioned on $\mathcal{B}_\Omega$. Therefore
   \begin{equation}\label{eq:tal1}
   E(f_n\tilde f_n | \mathcal{B}_\Omega) = E(\tilde f_n | \mathcal{B}_\Omega)E(f_n| \mathcal{B}_\Omega).
   \end{equation}
    Note also that
   \begin{equation}\label{eq:tal2}
   E[|f-\tilde{f_n}|] =  E\left[\big|(\theta\times\theta)^{-3n}f-(\theta\times\theta)^{-3n}\tilde{f_n}\big|\right] = E[|f-f_n|]<\epsilon.
   \end{equation}
    Write $\varphi_n=E(\tilde f_n | \mathcal{B}_\Omega),$ and $\tilde{\varphi_n}=E(f_n| \mathcal{B}_\Omega).$ By \eqref{eq:tal2} and the triangle inequality, $E[|\varphi-\varphi_n|]<\epsilon$ and $E[|\varphi-\tilde{\varphi_n}|]<\epsilon$. Therefore
   \begin{eqnarray*}
    E[f_n\tilde{f_n}]
    &=&E\big[ E[f_n\tilde f_n | \mathcal{B}_\Omega] \big]
    \mathop{=}^{\eqref{eq:tal1}} E[\varphi_n\tilde{\varphi_n}]=
    E[ (\varphi + \varphi_n - \varphi)( \varphi + \tilde{\varphi_n} -\varphi )]\\
    &=&\varphi^2 + \varphi E[(\varphi_n - \varphi)] +  \varphi E[(\tilde{\varphi_n}-\varphi)] + E[(\tilde{\varphi_n}-\varphi)(\tilde{\varphi_n}-\varphi )].
    \end{eqnarray*}
    (We used the fact that $\varphi$ is an a.s.\ constant and write it (notation abused) as a number.)
    Using the fact that all functions are bounded from above by $1$, their difference is bounded from above by $2$ and we have
    \begin{eqnarray*}
    | E[f_n\tilde{f_n}-\varphi^2] | \le E[|\varphi_n - \varphi|]+ E[|\tilde{\varphi_n}-\varphi|] +2 E[|(\tilde{\varphi_n}-\varphi )|] < 4\epsilon.
    \end{eqnarray*}
    As $E[f_n\tilde{f_n}]\to E[f^2]$ as $n\to\infty$, taking $n$ to infinity and then $\epsilon$ to zero
    yields $E[f^2]=E[\varphi^2]=E[\varphi]^2=E[f]^2$. Therefore $\mbox{var}(f)=0$ and $f$ is a $\mu\times\PP'$-a.s.\ constant.
    \end{proof}

\end{appendix}

\end{document}